\documentclass[a4paper,12pt,reqno]{amsart}
\usepackage{amsfonts}
\usepackage{amsmath}
\usepackage{amssymb}
\usepackage[a4paper]{geometry}
\usepackage{mathrsfs}
\usepackage{hyperref}
\renewcommand\eqref[1]{(\ref{#1})} 
\numberwithin{equation}{section}
\theoremstyle{plain}
\newtheorem{thm}{Theorem}[section]

\newtheorem{cor}[thm]{Corollary}
\newtheorem{lem}[thm]{Lemma}

\theoremstyle{definition}

\newtheorem{rem}[thm]{Remark}

\renewcommand{\wp}{\mathfrak S}
\newcommand{\Rn}{\mathbb R^{n}}

\begin{document}

   \title[Hardy and Rellich inequalities on homogeneous groups]
   {Hardy and Rellich inequalities, identities, and sharp remainders on homogeneous groups}

\author[Michael Ruzhansky]{Michael Ruzhansky}
\address{
  Michael Ruzhansky:
  \endgraf
  Department of Mathematics
  \endgraf
  Imperial College London
  \endgraf
  180 Queen's Gate, London SW7 2AZ
  \endgraf
  United Kingdom
  \endgraf
  {\it E-mail address} {\rm m.ruzhansky@imperial.ac.uk}
  }
\author[Durvudkhan Suragan]{Durvudkhan Suragan}
\address{
  Durvudkhan Suragan:
  \endgraf
  Institute of Mathematics and Mathematical Modelling
  \endgraf
  125 Pushkin str.
  \endgraf
  050010 Almaty
  \endgraf
  Kazakhstan
  \endgraf
  and
  \endgraf
  RUDN University
  \endgraf
  6 Miklukho-Maklay str., Moscow 117198
  \endgraf
  Russia
  \endgraf
  {\it E-mail address} {\rm suragan@math.kz}
  }

\thanks{The authors were supported in parts by the EPSRC
 grant EP/K039407/1 and by the Leverhulme Grant RPG-2014-02,
 as well as by the MESRK grant 5127/GF4. The second author was supported by the Ministry of Science of the Russian Federation (the Agreement number No 02.a03.21.0008).}


     \begin{abstract}
     We give sharp remainder terms of $L^{p}$ and weighted Hardy and Rellich inequalities
     on one of most general subclasses
     of nilpotent Lie groups, namely the class of homogeneous groups.  As consequences,
     we obtain analogues of the generalised classical
     Hardy and Rellich inequalities and the uncertainty principle on homogeneous groups.
     We also prove higher order inequalities of Hardy-Rellich type, all with sharp constants.
     A number of identities are derived including weighted and higher order types.
     \end{abstract}
     \maketitle

\section{Introduction}

In this paper we are interested in Hardy, Rellich, and higher order inequalities of Hardy-Rellich type in the setting of general homogeneous groups. Moreover, we are interested in questions of best constants, their attainability, and sharp expressions for the remainders.

\subsection{Hardy inequalities}
In the modern analysis of the $p$-Laplacian and for other problems, the $L^{p}$-Hardy inequality takes the form
\begin{equation}\label{HRn-p}
\left\|\frac{f(x)}{|x|}\right\|_{L^{p}(\Rn)} \leq \frac{p}{n-p}\left\| \nabla f\right\|_{L^{p}(\Rn)},\quad 1\leq p<n,
\end{equation}
where $\nabla$ is the standard gradient in $\mathbb{R}^{n}$, $f\in C_{0}^{\infty}(\mathbb{R}^{n})$,
and the constant $\frac{p}{n-p}$ is known to be sharp.

The one-dimensional version of this for $p=2$ was shown by Hardy in \cite{Hardy1919}, and then for other $p$ in \cite{Hardy:1920}, with Hardy partly attributing such a generalisation to Marcel Riesz in response to Hardy's demonstration of his inequality for $p=2$
to Riesz, see \cite{Hardy:1920} for the story behind these inequalities.

Inequality \eqref{HRn-p} has been intensively analysed in different setting, with different weights and remainder analysis, see e.g. Davies and Hinz \cite{Davies-Hinz}, and Davies \cite{Davies} for a review of the inequality and its numerous applications. We also refer to a recent interesting paper of Hoffmann-Ostenhof and Laptev \cite{Laptev15} on this subject for inequalities with weights, to \cite{HHLT-Hardy-many-particles} for many-particle versions, to \cite{EKL:Hardy-p-Lap} for $p$-Laplacian interpretations, and to many further references therein.  

In the analysis of sub-Laplacian and $p$-sub-Laplacian on e.g. homogeneous Carnot groups (or stratified groups) inequalities of this type have been also intensively investigated. In this case inequality \eqref{HRn-p} takes the form
\begin{equation}\label{HG-p}
\left\|\frac{f(x)}{d(x)}\right\|_{L^{p}(\mathbb G)} \leq \frac{p}{Q-p}\left\| \nabla_{H} f\right\|_{L^{p}(\mathbb G)},\quad Q\geq 3,\; 1<p<Q,
\end{equation}
where $Q$ is the homogeneous dimension of the homogeneous Carnot group $\mathbb G$, $\nabla_{H}$ is the horizontal gradient, and $d(x)$ is the so-called $\mathcal L$-gauge. This is a particular quasi-norm obtained from the fundamental solution of the sub-Laplacian: $d(x)$ is such that $d(x)^{2-Q}$ is a constant multiple of Folland's \cite{Folland-FS} fundamental solution of the sub-Laplacian on $\mathbb G$.

In the case of the Heisenberg group \eqref{HG-p} was proved for $p=2$ by Garofalo and Lanconelli \cite{GL}, see also D'Ambrosio \cite{DAmbrosio-Hardy}, and its extension to $p\not=2$ was obtained by Niu, Zhang and Wang \cite{NZW-Hardy-p}. Further extensions appeared by
Danielli, Garofalo and Phuc \cite{DGP-Hardy-potanal} on groups of Heisenberg type, on polarisable groups by Goldstein and Kombe \cite{GolKom},
and on Carnot groups by Jin and Shen \cite{Jin-Shen:Hardy-Rellich-AM-2011} and Lian \cite{Lian:Rellich}, together with certain weighted versions.

Thus, these papers establish \eqref{HG-p} and its weighted versions on different subclasses of stratified Lie groups with methods yielding also the sharp constant in the inequality.
One can observe that the $\mathcal L$-gauge $d(x)$ can be clearly replaced by another quasi-norm due to the equivalence of all homogeneous quasi-norms on stratified Lie groups but this may change the best constant in a way which is not easy to trace. Interestingly, in the setting of the Heisenberg group it was shown by Yang \cite{Yang:PAMS-Hardy-Heisenberg} that the $\mathcal L$-gauge $d(x)$ (sometimes also called the Koranyi-Folland or Kaplan gauge in this case) can be replaced by the Carnot-Carath\'eodory distance, and the inequality  \eqref{HG-p} remains valid with the same best constant $\frac{p}{Q-p}$.
A higher order extension of \eqref{HG-p} on stratified Lie groups was recently obtained by Ciatti, Cowling and Ricci \cite{Ciatti-Cowling-Ricci} for arbitrary homogeneous quasi-norms but the obtained constants are not optimal.

To finish our very incomplete literature survey of this topic, we only mention that such questions have been also considered on manifolds, see e.g. \cite{Grillo:Hardy-Rellich-PA-2003, Adimurthi-Sekar, DAmbrosio-Dipierro:Hardy-manifolds-AHP-2014,Kombe-Ozaydin:Hardy-Rellich-mfds}. Refinements including boundary terms over arbitrary domains have been obtained by the authors in \cite{Ruzhansky-Suragan:Layers} (see also \cite{JDE2017}, \cite{RS17b} and \cite{RSY17}), and further local weighted versions for sums of squares of vector fields of possibly limited regularity on manifolds were recently established in \cite{Ruzhansky-Suragan:Hardy}.

Hardy inequalities corresponding to higher order Grushin operators and to so-called $\Delta_{\lambda}$-Laplacians have been recently obtained by
Kogoj and Sonner \cite{Kogoj-Sonner:Hardy-lambda-CVEE-2016}.
These inequalities can be viewed as ones on particular types of homogeneous groups but with the homogeneous structure more general than that of the stratified groups.

The main aim of this paper is to establish versions of the Hardy and Rellich inequalities on general homogeneous groups. In fact, we also obtain sharp remainder terms in $L^{p}$ and weighted Hardy inequalities on homogeneous groups, and these equalities immediately imply Hardy's inequalities  by observing that these remainders are non\-negative. In addition, these precise equalities imply a number of other interesting inequalities as their consequences.

Homogeneous groups are Lie groups equipped with a family of
dilations compatible with the group law.
The abelian group $(\mathbb{R}^{n}; +)$, the Heisenberg group, homogeneous Carnot groups, stratified Lie groups, graded Lie groups are all special cases of the homogeneous groups. Analysis on homogeneous groups has been consistently developed by Folland and Stein \cite{FS-Hardy} and it presents a natural setting for the distillation of the results of harmonic analysis that depend only on the underlying group and dilation structures.
We note that homogeneous groups are nilpotent, and
the class of homogeneous groups gives almost the class of all nilpotent Lie groups but is not equal to it, see Dyer \cite{Dyer-1970} for an example of a (nine-dimensional) nilpotent Lie group that does not allow for any family of dilations.

Before giving an overview of our results let us mention the critical observation that the Hardy inequality \eqref{HRn-p} can be sharpened to the inequality
\begin{equation}\label{HRn-p-sh}
\left\|\frac{f(x)}{|x|}\right\|_{L^{p}(\Rn)} \leq \frac{p}{n-p}\left\| \frac{x}{|x|}\cdot\nabla f\right\|_{L^{p}(\Rn)},\quad
 1\leq p<n.
\end{equation}
It is clear that \eqref{HRn-p-sh} implies \eqref{HRn-p} since the function $\frac{x}{|x|}$ is bounded.
The remainder terms for \eqref{HRn-p-sh} have been analysed by Ioku, Ishiwata and Ozawa
\cite{IIO:Lp-Hardy}, see also Machihara, Ozawa and Wadade \cite{MOW:Hardy-Hayashi}.

As our first result, if $\mathbb G$ is a homogeneous group and $|\cdot|$ is a homogeneous quasi-norm on $\mathbb G$, as an analogue of \eqref{HRn-p-sh}
we obtain the following generalised $L^{p}$-Hardy
inequality:
\begin{equation}\label{iLpHardyeq0}
\left\|\frac{f}{|x|}\right\|_{L^{p}(\mathbb{G})}\leq\frac{p}{Q-p}\left\|\mathcal{R} f\right\|_{L^{p}(\mathbb{G})},
\quad 1<p<Q,
\end{equation}
for all complex-valued functions $f\in C_{0}^{\infty}(\mathbb{G}\backslash\{0\}).$
Here $\nabla=(X_{1},\ldots,X_{n})$ is a gradient on $\mathbb G$ with
$\{X_{1},\ldots,X_{n}\}$ a basis of the Lie algebra
 $\mathfrak{g}$ of $\mathbb{G}$,
$A$ is a $n$-diagonal matrix
\begin{equation}\label{EQ:mA0}
A={\rm diag}(\nu_{1},\ldots,\nu_{n}),
\end{equation}
where $\nu_{k}$ is the homogeneous degree of $X_{k}$, and
$$Q = {\rm Tr}\,A=\nu_{1}+\cdots+\nu_{n}$$
is the homogeneous dimension of $\mathbb G$.
We note that the exponential mapping ${\exp}_{\mathbb{G}}:\mathfrak g\to\mathbb G$ is a global diffeomorphism and the vector $e(x)=(e_{1}(x),\ldots,e_{n}(x))$
is the decomposition of its inverse ${\exp}_{\mathbb{G}}^{-1}$ with respect to the basis
$\{X_{1},\ldots,X_{n}\}$, namely, $e(x)$ is determined by
$${\exp}_{\mathbb{G}}^{-1}(x)=e(x)\cdot \nabla\equiv\sum_{j=1}^{n}e_{j}(x)X_{j}.$$
The inequalities of different types in this paper will follow from the corresponding identities: for example, for $p=2$ and $Q\geq 3$, the $L^{2}$-Hardy inequality \eqref{iLpHardyeq0} would follow from the identity
\begin{equation}\label{EQ:expL2-0}
\left\|\mathcal{R} f\right\|^{2}_{L^{2}(\mathbb{G})}=
\left(\frac{Q-2}{2}\right)^{2}\left\|\frac{f}{|x|}\right\|^{2}_{L^{2}(\mathbb{G})}+
\left\|\mathcal{R} f+\frac{Q-2}{2}\frac{f}{|x|}\right\|^{2}_{L^{2}(\mathbb{G})}.
\end{equation}
While we refer to Section \ref{SEC:2} for further details related to homogeneous Lie groups let us make a few remarks:
\begin{itemize}
\item In the abelian case of $\mathbb G=\mathbb{R}^{n}$ the Euclidean space,
we have $Q=n$, $e(x)=x$, $A=I$, and taking $|x|$ to be the Euclidean norm, \eqref{iLpHardyeq0}
gives \eqref{HRn-p-sh}.
\item Since $\mathbb G$ is a general homogeneous group, it does not have to be stratified or even graded. Therefore, the notion of a horizontal gradient does not make sense, and hence it is natural to work with the full gradient $\nabla$.
\item However, the gradient $\nabla$ is not homogeneous unless $\mathbb G$ is abelian. On the contrary, the operator
\begin{equation}\label{EQ:def-E}
\mathcal{R} :=\frac{d}{d|x|}
\end{equation}
is homogeneous of order $-1$ thus providing a natural analogue to the usual Euclidean gradient on $\mathbb{R}^{n}$ and to the radial derivative $\frac{x}{|x|}\cdot\nabla$ appearing in
\eqref{HRn-p-sh}.
\item In fact, the operator $\mathcal{R} $ can be interpreted precisely as the radial derivative on $\mathbb G$, see \eqref{dfdr}. The operator
\begin{equation}\label{EQ:def-Euler}
{\tt Euler}:=|x| \mathcal{R} 
\end{equation}
can be interpreted as the Euler type operator characterising the homogeneity on the homogeneous group $\mathbb G$:
$$
{\tt Euler}(f)=\nu f
 \; \textrm{ if and only if }\;
 f(D_r x)=r^{\nu} f(x)\;\; (\forall r>0, x\not=0),$$
where $D_r$ is the dilation on $\mathbb G$, see Lemma \ref{L:Euler}.
\item The constant $\frac{p}{Q-p}$ in \eqref{iLpHardyeq0} is sharp and is attained if and only if $f=0$. Contrary to \eqref{HG-p} where a particular choice of $d(x)$ is made, inequality \eqref{iLpHardyeq0} with the sharp constant holds true for any homogeneous quasi-norm on $\mathbb G$.
\item For $p=n$ or $p=Q$ the inequalities \eqref{HRn-p} and \eqref{iLpHardyeq0}
(and probably also \eqref{HG-p}) fail for any constant. The critical versions of \eqref{HRn-p}
with $p=n$ were investigated by Edmunds and Triebel \cite{ET-1999}, Adimurthi and Sandeep \cite{Adimurthi-Sandeep:PRSE-2002}, and of
\eqref{HRn-p-sh} by Ioku, Ishiwata and Ozawa \cite{IIO}.
Their generalisations as well as a number of other critical (logarithmic) Hardy inequalities on homogeneous groups were obtained in our recent preprint \cite{Ruzhansky-Suragan:critical}.
Here we only mention the range of logarithmic Hardy inequalities
\begin{equation}\label{LH2p}
\qquad  \underset{R>0}{\sup}\left\|\frac{f-f_{R}}{|x|^{\frac{Q}{p}}{\log}\frac{R}{|x|}}
\right\|_{L^{p}(\mathbb{G})}\leq
\frac{p}{p-1}\left\| \frac{1}{|x|^{\frac{Q}{p}-1}}\mathcal{R} f\right\|_{L^{p}(\mathbb{G})},
\end{equation}
for all $1< p<\infty$, where $f_{R}=f(R\frac{x}{|x|})$. We refer to
\cite{Ruzhansky-Suragan:critical} for further explanations and extensions but only mention here that for $p=Q$ the inequality \eqref{LH2p} gives the critical case of Hardy's inequalities \eqref{iLpHardyeq0}.
\end{itemize}

\subsection{Rellich inequalities}
At the same time the Rellich inequalities give results of the type
$$
\int_{\Rn}\frac{|f|^{p}}{|x|^{\alpha}} dx \leq C \int_{\Rn} \frac{|\Delta f|^{p}}{|x|^{\beta}} dx
$$
for certain relations between $\alpha,\beta, n,p$.
For example, the classical result by Rellich appearing at the 1954 ICM in Amsterdam
\cite{Rellich:ineq-1956}
stated the inequality
\begin{equation}\label{EQ:Rellich1}
\left\|\frac{f}{|x|^{2}}\right\|_{L^{2}(\Rn)}\leq \frac{4}{n(n-4)}\|\Delta f\|_{L^{2}(\Rn)},\quad n\geq 5.
\end{equation}
We refer e.g. to Davies and Hinz
\cite{Davies-Hinz} for history and further extensions,
including the derivation of sharp constants, and to
\cite{Kombe:Rellich-Carnot-2010} and \cite{Lian:Rellich} for the corresponding results for the
sub-Laplacian on homogeneous Carnot groups.

To establish an analogue of \eqref{EQ:Rellich1} in the setting of homogeneous groups is an interesting question. The first obstacle to it is that since the group does not have to be stratified and not even graded, there may be no homogeneous left-invariant hypoelliptic differential operators on $\mathbb G$ at all, to enter the right hand side of \eqref{EQ:Rellich1}.

However, similar to the discussion of the Hardy inequality before, the inequality \eqref{EQ:Rellich1} can be expressed in terms of the radial derivative $\partial_{r}=\frac{x}{|x|}\cdot\nabla$ in the form
\begin{equation}\label{EQ:Rellich2}
\left\|\frac{f}{|x|^{2}}\right\|_{L^{2}(\Rn)}\leq \frac{4}{n(n-4)}\left\|\partial_{r}^{2} f+\frac{n-1}{|x|}\partial_{r}f\right\|_{L^{2}(\Rn)},\quad n\geq 5.
\end{equation}
Similarly to our Hardy inequality \eqref{iLpHardyeq0}, we show the analogue of this for general homogeneous groups using the operator $\mathcal{R} $ from \eqref{EQ:def-E}:
\begin{equation}\label{iwHardyeq-R}
\left\|\frac{f}{|x|^{2}}
\right\|_{L^{2}(\mathbb{G})}\leq\frac{4}{Q(Q-4)}\left\|\mathcal{R} ^{2} f+\frac{Q-1}{|x|}\mathcal{R}  f
\right\|_{L^{2}(\mathbb{G})},\quad Q\geq 5,
\end{equation}
for all complex-valued functions $f\in C_{0}^{\infty}(\mathbb{G}\backslash\{0\}).$

As discussed for the Hardy inequality, the expression on the right hand side of \eqref{iwHardyeq-R} appears to be natural since there is no analogue of homogeneous Laplacian or sub-Laplacian on general homogeneous groups to formulate a version similar to \eqref{EQ:Rellich1}. In fact, there may be no homogeneous hypoelliptic left-invariant differential operators at all: the existence of such an operator would imply that the group must be graded as was shown by Miller \cite{Miller:graded} with a small gap in the proof later corrected by ter Elst and Robinson \cite{ter-Elst-Robinson:Rockland}, see \cite[Proposition 4.1.3]{FR} for a simple proof.

Again, we show that for every homogeneous quasi-norm the constant $\frac{4}{Q(Q-4)}$ in \eqref{iwHardyeq-R} is sharp and the equality is reached only for $f=0$.

\subsection{Higher order Hardy-Rellich inequalities}

In \cite{Davies-Hinz}, Davies and Hinz established a number of higher order Rellich inequalities of the form
$$
\int_{\Rn}\frac{|f|^{p}}{|x|^{\alpha}} dx \leq C \int_{\Rn} \frac{|\Delta^m f|^{p}}{|x|^{\beta}} dx
$$
for certain relations between $\alpha,\beta, n, m, p$, for all $f\in C^{\infty}_{0}(\Rn\backslash \{0\})$. By combining this inequality with the Hardy inequality one obtains the inequality with $\nabla \Delta^{m}f$ on the right hand side. They have been obtained by iterating Rellich's inequality and, most surprisingly, this method yielded the sharp constants as well. Their method, however, does not seem to be readily extendible to general homogeneous groups.

In this paper we adopt an approach different from that of Davies and Hinz \cite{Davies-Hinz}. Namely, for $p=2$, we can iterate the exact representation formulae for the remainder that we obtained for the Rellich inequality and for the weighted Hardy inequalities. This yields higher order remainders that can be then also used to argue the sharpness of the constant.
Thus, in analogy to the operators appearing so far,
for $Q\geq 3$ and $\prod_{j=0}^{k-1}
\left|\frac{Q-2}{2}-(\alpha+j)\right|\neq0$,
we prove the inequality
\begin{equation}\label{awHardyeq-high-R}
\left\|\frac{f}{|x|^{k+\alpha}}
\right\|_{L^{2}(\mathbb{G})}\leq
\left[\prod_{j=0}^{k-1}
\left|\frac{Q-2}{2}-(\alpha+j)\right|\right]^{-1}\left\|\frac{1}{|x|^{\alpha}}\mathcal{R} ^{k}f\right\|_{L^{2}(\mathbb{G})},
\end{equation}
where the appearing constant is sharp and is attained if and only if $f=0$.

For $k=1$, this gives weighted Hardy's inequality \eqref{iLpHardyeq0} (in addition, it gives Hardy's inequality for $\alpha=0$).
For $k=2$ this can be thought of as a Hardy-Rellich type inequality, while for larger $k$ this corresponds to higher order Rellich inequalities.

\subsection{Review of results}
In the case of $\Rn$ expressions for the remainder terms in Hardy and Rellich inequalities have been recently analysed in \cite{IIO:Lp-Hardy, MOW:Hardy-Hayashi,MOW-Rellich}.
In this paper we obtain representation formulae by making use of the operator $\mathcal{R} $ described above, on general homogeneous groups. Moreover, in Section \ref{Sec4} and Section \ref{SEC:ho} we also obtain weighted representation formulae and Hardy-type inequalities which we believe to be new already in the Euclidean setting of $\Rn$.

Thus, in this paper we show that for a homogeneous group $\mathbb{G}$ of homogeneous dimension $Q\geq 2$ and any homogeneous quasi-norm $|\cdot|$ we have the following equalities and inequalities providing analogues of Hardy and Rellich inequalities on homogeneous groups:

\begin{itemize}
\item
Let $f\in C_{0}^{\infty}(\mathbb{G}\backslash\{0\})$ be a real-valued function and let $1<p<Q$.
Then
\begin{multline}\label{iLH2}
\qquad
\left\|\frac{p}{Q-p}\mathcal{R} f\right\|^{p}_{L^{p}(\mathbb{G})}
-\left\|\frac{f}{|x|}\right\|^{p}_{L^{p}(\mathbb{G})} \\
=
p\int_{\mathbb{G}}I_{p}\left(\frac{f}{|x|},-\frac{p}{Q-p}\mathcal{R} f\right)\left|\frac{f}{|x|}+\frac{p}{Q-p}\mathcal{R} f\right|^{2}dx,
\end{multline}
where $\mathcal{R}$
is defined in \eqref{EQ:def-E}
and $I_{p}$ is given by
$$
I_{p}(h,g)=(p-1)\int_{0}^{1}|\xi h+(1-\xi)g|^{p-2}\xi d\xi.
$$
\item
Since $I_{p}\geq 0$ is nonnegative, the identity \eqref{iLH2} implies the generalised $L^{p}$-Hardy
inequality
\begin{equation}\label{iLpHardyeq}
\left\|\frac{f}{|x|}\right\|_{L^{p}(\mathbb{G})}\leq\frac{p}{Q-p}\left\|\mathcal{R} f\right\|_{L^{p}(\mathbb{G})},
\quad 1<p<Q,
\end{equation}
for all real-valued function $f\in C_{0}^{\infty}(\mathbb{G}\backslash\{0\}).$
From this we can pass to complex-valued functions by a routine argument but we will also give an independent proof of \eqref{iLpHardyeq}.
\item
As a consequence of \eqref{iLpHardyeq} we obtain the following generalised uncertainty principle:
\begin{equation}\label{iUP1}
\quad \left(\int_{\mathbb{G}}\left|\mathcal{R} f\right|^{p}dx\right)^{\frac{1}{p}}
 \left(\int_{\mathbb{G}}|x|^{q}
|f|^{q}dx\right)^{\frac{1}{q}}
\geq\frac{Q-p}{p}\int_{\mathbb{G}}
|f|^{2}dx,\quad \frac{1}{p}+\frac{1}{q}=1,
\end{equation}
for all complex-valued functions $f\in C^{\infty}_{0}(\mathbb{G})$.
\item
For every complex-valued function $f\in C^{\infty}_{0}(\mathbb{G}\backslash\{0\})$ and $\alpha\in\mathbb{R}$
we have the following weighted identity:
\begin{multline}\label{iawH}
\qquad \left\|\frac{1}{|x|^{\alpha}}\mathcal{R} f\right\|^{2}_{L^{2}(\mathbb{G})}=
\left(\frac{Q-2}{2}-\alpha\right)^{2}\left\|\frac{f}{|x|^{\alpha+1}}
\right\|^{2}_{L^{2}(\mathbb{G})}
+\left\|\frac{1}{|x|^{\alpha}}\mathcal{R} f+\frac{Q-2-2\alpha}{2|x|^{\alpha+1}}f
\right\|^{2}_{L^{2}(\mathbb{G})}.
\end{multline}
\item
Identity \eqref{iawH} implies different estimates. For example, for
$\alpha=1$ it gives the following generalised weighted Hardy
inequality:
\begin{equation}\label{awHardyeq}
\left\|\frac{f}{|x|^{2}}
\right\|_{L^{2}(\mathbb{G})}\leq
\frac{2}{Q-4}\left\|\frac{1}{|x|}\mathcal{R} f\right\|_{L^{2}(\mathbb{G})},\quad Q\geq 5,
\end{equation}
for all complex-valued $f\in C_{0}^{\infty}(\mathbb{G}\backslash\{0\})$.
Similarly, for any $\alpha\in\mathbb{R}$ such that $Q-2\alpha-2\neq0$
we have
\begin{equation}\label{awHardyeq-g0}
\qquad \left\|\frac{f}{|x|^{\alpha+1}}\right\|_{L^{2}(\mathbb{G})}\leq
\frac{2}{|Q-2-2\alpha|}
\left\|\frac{1}{|x|^{\alpha}}\mathcal{R} f\right\|_{L^{2}(\mathbb{G})},
\end{equation}
where the constant in the right hand side is sharp, see
Corollary \ref{waHardy}.

\item
Let $Q\geq 5$.
Then for every complex-valued function $f\in C^{\infty}_{0}(\mathbb{G})$
we have the identity
\begin{multline}\label{iwH}
\qquad \left\|\mathcal{R} ^{2} f+\frac{Q-1}{|x|}\mathcal{R}  f
+\frac{Q(Q-4)}{4|x|^{2}}f\right\|^{2}_{L^{2}(\mathbb{G})}
+\frac{Q(Q-4)}{2}\left\|\frac{1}{|x|}\mathcal{R}  f
+\frac{Q-4}{2|x|^{2}}f\right\|^{2}_{L^{2}(\mathbb{G})}\\
=\left\|\mathcal{R} ^{2} f+\frac{Q-1}{|x|}\mathcal{R}  f
\right\|^{2}_{L^{2}(\mathbb{G})}-\left(\frac{Q(Q-4)}{4}\right)^{2}\left\|\frac{f}{|x|^{2}}
\right\|^{2}_{L^{2}(\mathbb{G})}.
\end{multline}
\item
Simply by dropping two positive terms, \eqref{iwH} implies the Rellich inequality
\begin{equation}\label{iwHardyeq}
\left\|\frac{f}{|x|^{2}}
\right\|_{L^{2}(\mathbb{G})}\leq\frac{4}{Q(Q-4)}\left\|\mathcal{R} ^{2} f+\frac{Q-1}{|x|}\mathcal{R}  f
\right\|_{L^{2}(\mathbb{G})}, \quad Q\geq 5,
\end{equation}
for all complex-valued functions $f\in C_{0}^{\infty}(\mathbb{G}\backslash\{0\}).$

\item For $Q\geq 3$, $\alpha\in\mathbb{R}$ and $k\in\mathbb N$ such that $\prod_{j=0}^{k-1}
\left|\frac{Q-2}{2}-(\alpha+j)\right|\neq0$ we have
the higher order Hardy-Rellich type inequalities:
\begin{equation}\label{awHardyeq-high-R2}
\left\|\frac{f}{|x|^{k+\alpha}}
\right\|_{L^{2}(\mathbb{G})}\leq
\left[\prod_{j=0}^{k-1}
\left|\frac{Q-2}{2}-(\alpha+j)\right|\right]^{-1}\left\|\frac{1}{|x|^{\alpha}}\mathcal{R} ^{k}f\right\|_{L^{2}(\mathbb{G})},
\end{equation}
for all complex-valued functions $f\in C_{0}^{\infty}(\mathbb{G}\backslash\{0\})$,
where the appearing constant is sharp and is attained if and only if $f=0$.

It is interesting to note that in particular, taking $\alpha=0$, if $Q$ is even, inequality \eqref{awHardyeq-high-R2} becomes
\begin{equation}\label{awHardyeq-high-R2-nw}
\left\|\frac{f}{|x|^{k}}
\right\|_{L^{2}(\mathbb{G})}\leq
\left[\prod_{j=0}^{k-1}
\left(\frac{Q-2}{2}-j\right)\right]^{-1}\left\|\mathcal{R} ^{k}f\right\|_{L^{2}(\mathbb{G})},
\quad \forall\; k<\frac{Q}{2},\; Q \textrm{ even}.
\end{equation}
If $Q$ is odd, since $\frac{Q}{2}$ is not an integer, there is no restriction $k$, so that we have
\begin{equation}\label{awHardyeq-high-R2-nw2}
\left\|\frac{f}{|x|^{k}}
\right\|_{L^{2}(\mathbb{G})}\leq
\left[\prod_{j=0}^{k-1}
\left|\frac{Q-2}{2}-j\right|\right]^{-1}\left\|\mathcal{R} ^{k}f\right\|_{L^{2}(\mathbb{G})}.
\quad \forall\; k\in\mathbb N,\; Q \textrm{ odd},
\end{equation}
The constants in \eqref{awHardyeq-high-R2-nw} and \eqref{awHardyeq-high-R2-nw2} are sharp.

\item Moreover, 
we derive the exact formula for the remainder, i.e.
for the difference between the right and left hand sides in \eqref{awHardyeq-high-R2}, for any
$k\in\mathbb N$ and $\alpha\in\mathbb{R}$,
\begin{multline}\label{equality-high-rem0}
\qquad \left\|\frac{1}{|x|^{\alpha}}\mathcal{R} ^{k}f\right\|^{2}_{L^{2}(\mathbb{G})}-
\left[\prod_{j=0}^{k-1}
\left(\frac{Q-2}{2}-(\alpha+j)\right)^{2}\right]\left\|\frac{f}{|x|^{k+\alpha}}
\right\|^{2}_{L^{2}(\mathbb{G})}
\\=\sum_{l=1}^{k-1}\left[\prod_{j=0}^{l-1}
\left(\frac{Q-2}{2}-(\alpha+j)\right)^{2}\right]\left\|\frac{1}{|x|^{l+\alpha}}\mathcal{R} ^{k-l}f+
\frac{Q-2(l+1+\alpha)}{2|x|^{l+1+\alpha}}\mathcal{R} ^{k-l-1}f\right\|^{2}_{L^{2}(\mathbb{G})}
\\
+\left\|\frac{1}{|x|^{\alpha}}\mathcal{R} ^{k}f+\frac{Q-2-2\alpha}{2|x|^{1+\alpha}}\mathcal{R} ^{k-1}f \right\|^{2}_{L^{2}(\mathbb{G})}.
\end{multline}
\end{itemize}

In Section \ref{SEC:2} we very briefly review the main concepts of homogeneous
groups and fix the notation. In Section \ref{Sec3} we give sharp remainder terms of $L^{p}$-Hardy inequality on homogeneous groups and, as consequences, we obtain analogues of the classical
Hardy inequality as well as uncertainty principle on homogeneous groups. In Section \ref{Sec4}  and Section \ref{Sec5} sharp remainder terms of weighted Hardy and Rellich inequalities on homogeneous groups are studied, respectively. In Section \ref{SEC:ho} we present higher order Hardy-Rellich type inequalities.

\section{Preliminaries}
\label{SEC:2}

Here we very briefly recall some basics of the analysis on homogeneous groups and establish some properties of the operator $\mathcal{R} $ from \eqref{EQ:def-E}.
For the general background details on homogeneous groups we refer to the book
\cite{FS-Hardy} by Folland and Stein as well as to the recent monograph \cite{FR} by V. Fischer and the first named author.

We recall that a family of dilations of a Lie algebra $\mathfrak{g}$
is a family of linear mappings of the form
$$D_{\lambda}={\rm Exp}(A \,{\rm ln}\lambda)=\sum_{k=0}^{\infty}
\frac{1}{k!}({\rm ln}(\lambda) A)^{k},$$
where $A$ is a diagonalisable linear operator on $\mathfrak{g}$
with positive eigenvalues,
and each $D_{\lambda}$ is a morphism of the Lie algebra $\mathfrak{g}$,
that is, a linear mapping
from $\mathfrak{g}$ to itself which respects the Lie bracket:
$$\forall X,Y\in \mathfrak{g},\, \lambda>0,\;
[D_{\lambda}X, D_{\lambda}Y]=D_{\lambda}[X,Y].$$
A {\em homogeneous group} is a connected simply connected Lie group whose
Lie algebra is equipped with dilations.
Homogeneous groups are necessarily nilpotent and hence, in particular, the exponential mapping $\exp_{\mathbb G}:\mathfrak g\to\mathbb G$ is a global diffeomorphism.
It induces the dilation structure on $\mathbb G$ which we continue to denote by $D_{\lambda}x$ or simply by $\lambda x$. We denote by
$$Q := {\rm Tr}\,A$$
the homogeneous dimension of $\mathbb G$.

Let $dx$ denote the Haar measure on $\mathbb{G}$ and let $|S|$ denote the corresponding volume of a measurable set $S\subset \mathbb{G}$.
Then we have
\begin{equation}
|D_{\lambda}(S)|=\lambda^{Q}|S| \quad {\rm and}\quad \int_{\mathbb{G}}f(\lambda x)
dx=\lambda^{-Q}\int_{\mathbb{G}}f(x)dx.
\end{equation}

We now fix a basis $\{X_{1},\ldots,X_{n}\}$ of $\mathfrak{g}$
such that
$$AX_{k}=\nu_{k}X_{k}$$
for each $k$, so that $A$ can be taken to be
$A={\rm diag} (\nu_{1},\ldots,\nu_{n}).$
Then each $X_{k}$ is homogeneous of degree $\nu_{k}$ and also
$$
Q=\nu_{1}+\cdots+\nu_{n}.
$$
The decomposition of ${\exp}_{\mathbb{G}}^{-1}(x)$ in the Lie algebra $\mathfrak g$ defines the vector
$$e(x)=(e_{1}(x),\ldots,e_{n}(x))$$
by the formula
$${\exp}_{\mathbb{G}}^{-1}(x)=e(x)\cdot \nabla\equiv\sum_{j=1}^{n}e_{j}(x)X_{j},$$
where $\nabla=(X_{1},\ldots,X_{n})$.
Alternatively, this means the equality
$$x={\exp}_{\mathbb{G}}\left(e_{1}(x)X_{1}+\ldots+e_{n}(x)X_{n}\right).$$
By homogeneity this implies
$$rx={\exp}_{\mathbb{G}}\left(r^{\nu_{1}}e_{1}(x)X_{1}+\ldots
+r^{\nu_{n}}e_{n}(x)X_{n}\right),$$
that is,
$$
e(rx)=(r^{\nu_{1}}e_{1}(x),\ldots,r^{\nu_{n}}e_{n}(x)).
$$
Consequently, we can calculate
\begin{align*}
\frac{d}{dr}f(rx) & =  \frac{d}{dr}f({\exp}_{\mathbb{G}}
\left(r^{\nu_{1}}e_{1}(x)X_{1}+\ldots
+r^{\nu_{n}}e_{n}(x)X_{n}\right)) \\
& =  \left[(\nu_{1}r^{\nu_{1}-1}e_{1}(x)X_{1}+\ldots
+\nu_{n}r^{\nu_{n}-1}e_{n}(x)X_{n})f\right](rx).
\end{align*}
this yields the equality
\begin{equation}\label{dfdr}
\frac{d}{dr}f(rx)=\mathcal{R} f(rx).
\end{equation}
In other words, the operator $\mathcal{R} $ plays the role of the radial derivative on $\mathbb G$.
It follows from \eqref{EQ:def-Euler} that $\mathcal{R} $ is homogeneous of order $-1$.

The following relation between $\mathcal{R} $ and Euler's operator in \eqref{EQ:def-Euler}
will be of importance to us.
\begin{lem}\label{L:Euler}
Define the operator
\begin{equation}\label{EQ:def-Euler}
{\tt Euler}:=|x| \mathcal{R}.
\end{equation}
If $f:\mathbb G\backslash \{0\}\to\mathbb R$ is continuously differentiable, then
$$
{\tt Euler}(f)=\nu f
 \; \textrm{ if and only if }\;
 f(D_{r} x)=r^{\nu} f(x)\;\; (\forall r>0, x\not=0).$$
\end{lem}
\begin{proof}
If $f$ is positively homogeneous of order $\nu$, i.e. if $f(r x)=r^{\nu}f(x)$ holds for all $r>0$ and $x\not=0$, then applying \eqref{dfdr} to such $f$ and setting $r=1$ we get
$$
{\tt Euler}(f)=\nu f.
$$
Conversely, let us fix $x\not=0$ and define $g(r):=f(rx)$.
Using \eqref{dfdr}, the equality ${\tt Euler}(f)(rx)=\nu f(rx)$ means that
$$
g'(r)=\frac{d}{dr}f(rx)=\frac{1}{r} {\tt Euler}(f)(rx)=\frac{\nu}{r} f(rx)=\frac{\nu}{r}g(r).
$$
Consequently, $g(r)=g(1) r^{\nu}$, i.e. $f(rx)=r^{\nu} f(x)$ and thus $f$ is positively homogeneous of order $r$.
\end{proof}

A {\em homogeneous quasi-norm} on a homogeneous group $\mathbb G$ is
a continuous non-negative function
$$\mathbb{G}\ni x\mapsto |x|\in [0,\infty),$$
satisfying the properties

\begin{itemize}
\item   $|x^{-1}| = |x|$ for all $x\in \mathbb{G}$,
\item  $|\lambda x|=\lambda |x|$ for all
$x\in \mathbb{G}$ and $\lambda >0$,
\item  $|x|= 0$ if and only if $x=0$.
\end{itemize}

The following polar decomposition will be useful for our analysis:
there is a (unique)
positive Borel measure $\sigma$ on the
unit sphere
\begin{equation}\label{EQ:sphere}
\wp:=\{x\in \mathbb{G}:\,|x|=1\},
\end{equation}
such that for all $f\in L^{1}(\mathbb{G})$ we have
\begin{equation}\label{EQ:polar}
\int_{\mathbb{G}}f(x)dx=\int_{0}^{\infty}
\int_{\wp}f(ry)r^{Q-1}d\sigma(y)dr.
\end{equation}
We refer to Folland and Stein \cite{FS-Hardy} for the proof, which can be also found in
\cite[Section 3.1.7]{FR}.

\section{$L^{p}$-Hardy inequality and uncertainty principle}
\label{Sec3}

In this section and in the sequel we adopt all the notation introduced in Section \ref{SEC:2} concerning homogeneous groups and the operator $\mathcal{R} $.

We now present the $L^{p}$-Hardy inequality and the remainder formula on the homogeneous group $\mathbb{G}$.

\begin{thm}\label{RemHardy}
Let $\mathbb{G}$ be a homogeneous group
of homogeneous dimension $Q$. Let $|\cdot|$ be a homogeneous quasi-norm on $\mathbb{G}$.
Let $1<p<Q$.

\begin{itemize}
\item[(i)] Let $f\in C_{0}^{\infty}(\mathbb{G}\backslash\{0\})$ be a complex-valued function.
Then
\begin{equation}\label{LpHardyinC}
\left\|\frac{f}{|x|}\right\|_{L^{p}(\mathbb{G})}\leq\frac{p}{Q-p}\left\|\mathcal{R} f\right\|_{L^{p}(\mathbb{G})},\quad1<p<Q.
\end{equation}
The constant $\frac{p}{Q-p}$ is sharp. Moreover, the equality in \eqref{LpHardyinC} is attained if and only if $f=0$.

\item[(ii)] Let $f\in C_{0}^{\infty}(\mathbb{G}\backslash\{0\})$ be a real-valued function.
Setting
$$u:=u(x)=-\frac{p}{Q-p}\mathcal{R} f(x)$$
and
$$v:=v(x)=\frac{f(x)}{|x|},$$
we have the identity
\begin{equation}\label{LH2}
\qquad
\left\|u\right\|^{p}_{L^{p}(\mathbb{G})}
-\left\|v\right\|^{p}_{L^{p}(\mathbb{G})}=p\int_{\mathbb{G}}I_{p}(v,u)|v-u|^{2}dx,
\end{equation}
where
$$
I_{p}(h,g)=(p-1)\int_{0}^{1}|\xi h+(1-\xi)g|^{p-2}\xi d\xi.
$$
\item[(iii)]
In the case $p=2$, the identity \eqref{LH2} holds for complex-valued functions as well.
Namely, if $f\in C_{0}^{\infty}(\mathbb{G}\backslash\{0\})$ is a complex-valued function then
we have
\begin{equation}\label{EQ:expL2}
\left\|\mathcal{R} f\right\|^{2}_{L^{2}(\mathbb{G})}=
\left(\frac{Q-2}{2}\right)^{2}\left\|\frac{f}{|x|}\right\|^{2}_{L^{2}(\mathbb{G})}+
\left\|\mathcal{R} f+\frac{Q-2}{2}\frac{f}{|x|}\right\|^{2}_{L^{2}(\mathbb{G})},\quad Q\geq 3.
\end{equation}
\end{itemize}
\end{thm}

\begin{rem}\label{LpHardy}
Let us show that Part (ii) implies Part (i).
Since the right hand side of \eqref{LH2} is nonnegative it follows that
\begin{equation}\label{LpHardyeq}
\left\|\frac{f}{|x|}\right\|_{L^{p}(\mathbb{G})}\leq\frac{p}{Q-p}\left\|\mathcal{R} f\right\|_{L^{p}(\mathbb{G})},\quad 1<p<Q,
\end{equation}
for all real-valued $f\in C_{0}^{\infty}(\mathbb{G}\backslash\{0\}).$
Consequently, the same inequality follows for all complex-valued functions by using the identity (cf. Davies \cite[p. 176]{Davies-bk:Semigroups-1980})
\begin{equation}\label{EQ:Davies-rc}
\forall z\in\mathbb C:\;
|z|^{p}=\left(\int_{-\pi}^{\pi}|\cos\theta|^{p} d\theta\right)^{-1}
\int_{-\pi}^{\pi}\left| {\rm Re}(z)\cos\theta+{\rm Im}(z)\sin\theta\right|^{p}d\theta,
\end{equation}
which follows from the representation $z=r(\cos\phi+i\sin\phi)$ by some manipulations.

Thus, we obtain inequality \eqref{LpHardyinC}, also implying that the constant $\frac{p}{Q-p}$ is sharp. We now claim that this constant is attained only for $f=0$. Formula \eqref{EQ:Davies-rc} shows that this needs to be checked only for real-valued functions $f$. From the identity \eqref{LH2}, if its right hand side is zero, we must have $u=v$, that is,
$$
-\frac{p}{Q-p}\mathcal{R} f(x)=\frac{f(x)}{|x|}.
$$
But this means that ${\tt Euler}(f)=-\frac{Q-p}{p}f$. By Lemma \ref{L:Euler} it implies that $f$ is positively homogeneous of order $-\frac{Q-p}{p}$, i.e. there exists a function $h:\wp\to\mathbb C$ such that
\begin{equation}\label{EQ:homr1}
f(x)=|x|^{-\frac{Q-p}{p}}h\left(\frac{x}{|x|}\right),
\end{equation}
where $\wp$ is the sphere from \eqref{EQ:sphere}.
In particular this implies that $f$ can not be compactly supported unless it is zero.
\end{rem}

\begin{rem}\label{LpHardy2}
We also note that the statement of the theorem can be slightly extended in the following way.
We denote by $H^{1}_{\mathcal{R} }(\mathbb G)$ the space of all $f\in L^{2}(\mathbb G)$ such that $\mathcal{R} f \in L^{2}(\mathbb G)$. Then the statement of Theorem \ref{RemHardy} remains true for functions in $H^{1}_{\mathcal{R}}(\mathbb G)$. Indeed, the proof of \eqref{LpHardyinC} given below works equally well for such functions, and more general analysis of these issues will appear elsewhere. For the sharpness and the equality in \eqref{LpHardyinC}, having \eqref{EQ:homr1} also implies that
$\frac{f(x)}{|x|}=|x|^{-\frac{Q}{p}}h\left(\frac{x}{|x|}\right)$ is not in $L^{p}(\mathbb G)$ unless $h=0$ and $f=0$.
\end{rem}

Thus, Remark \ref{LpHardy} shows that Part (ii), namely the representation formula \eqref{LH2}, implies Part (i) of the theorem. So, we only need to prove Part (ii). However, we now also give an independent proof of \eqref{LpHardyinC} for complex-valued functions without relying on the formula \eqref{EQ:Davies-rc}, especially since this calculation will be also useful in proving Part (ii).

\begin{proof}[Proof of Theorem \ref{RemHardy}]

Introducing polar coordinates $(r,y)=(|x|, \frac{x}{\mid x\mid})\in (0,\infty)\times\wp$ on $\mathbb{G}$, where $\wp$ is the sphere in \eqref{EQ:sphere}, and using the integration formula \eqref{EQ:polar}
one calculates
$$
\int_{\mathbb{G}}
\frac{|f(x)|^{p}}
{|x|^{p}}dx
=\int_{0}^{\infty}\int_{\wp}
\frac{|f(ry)|^{p}}
{r^{p}}r^{Q-1}d\sigma(y)dr
$$
$$
=-\frac{p}{Q-p}\int_{0}^{\infty} r^{Q-p} \,{\rm Re}\int_{\wp}
|f(ry)|^{p-2} f(ry) \overline{\frac{df(ry)}{dr}}d\sigma(y)dr
$$
\begin{equation}\label{EQ:formula1}
=-\frac{p}{Q-p} {\rm Re}\int_{\mathbb{G}}
\frac{|f(x)|^{p-2}f(x)}{|x|^{p-1}}
\overline{\frac{d}{d|x|}f(x)}dx.
\end{equation}
From this using the H\"older inequality for $\frac{1}{q}+\frac{1}{p}=1$ we get
$$\int_{\mathbb{G}}
\frac{|f(x)|^{p}}
{|x|^{p}}dx
=-\frac{p}{Q-p} {\rm Re}\int_{\mathbb{G}}
\frac{|f(x)|^{p-2}f(x)}{|x|^{p-1}}
\overline{\frac{d}{d|x|}f(x)}dx
$$
$$
\leq \frac{p}{Q-p}\left(\int_{\mathbb{G}}
\left|\frac{|f(x)|^{p-2}f(x)}{|x|^{p-1}}\right|^{q}dx\right)^{\frac{1}{q}} \left(\int_{\mathbb{G}}
\left|\frac{d}{d|x|}f(x)\right|^{p}dx\right)^{\frac{1}{p}}
$$
$$
=\frac{p}{Q-p} \left(\int_{\mathbb{G}}
\frac{|f(x)|^{p}}{|x|^{p}}dx\right)^{1-\frac{1}{p}}
\left\|\mathcal{R}
f(x)\right\|_{L^{p}(\mathbb{G})}.
$$
This proves inequality \eqref{LpHardyinC} in Part (i).

\medskip
Let us now prove Part (ii). Using notations
$$u:=u(x)=-\frac{p}{Q-p}\mathcal{R} f,$$
and
$$v:=v(x)=\frac{f}{|x|},$$
formula \eqref{EQ:formula1}
can be restated as
\begin{equation}\label{vu2}
\|v\|_{L^{p}(\mathbb{G})}^{p}={\rm Re}\int_{\mathbb{G}}|v|^{p-2}v \overline{u} dx.
\end{equation}
In the case of a real-valued $f$ formula \eqref{EQ:formula1} becomes
$$
\int_{\mathbb{G}}
\frac{|f(x)|^{p}}
{|x|^{p}}dx
=-\frac{p}{Q-p} \int_{\mathbb{G}}
\frac{|f(x)|^{p-2}f(x)}{|x|^{p-1}}
\frac{df(x)}{d|x|}dx
$$
and \eqref{vu2} becomes
\begin{equation}\label{vu}
\|v\|_{L^{p}(\mathbb{G})}^{p}=\int_{\mathbb{G}}|v|^{p-2}v u dx.
\end{equation}
On the other hand, for any $L^{p}$-integrable real-valued functions $u$ and $v$, one has
\begin{multline*}
\|u\|_{L^{p}(\mathbb{G})}^{p}-\|v\|_{L^{p}(\mathbb{G})}^{p}
+p\int_{\mathbb{G}}(|v|^{p}-|v|^{p-2}v u)dx \\
=\int_{\mathbb{G}}(|u|^{p}+(p-1)|v|^{p}-p|v|^{p-2}vu)dx=p\int_{\mathbb{G}}I_{p}(v,u)|v-u|^{2}dx,
\end{multline*}
where
$$
I_{p}(v,u)=(p-1)\int_{0}^{1}|\xi v+(1-\xi)u|^{p-2}\xi d\xi.
$$
Combining this with \eqref{vu} we obtain
$$\left\|u\right\|^{p}_{L^{p}(\mathbb{G})}
-\left\|v\right\|^{p}_{L^{p}(\mathbb{G})}=p\int_{\mathbb{G}}I_{p}(v,u)|v-u|^{2}dx.
$$
The equality \eqref{LH2} is proved.

\medskip
Let us now prove Part (iii). If $p=2$, the identity \eqref{vu2} takes the form
\begin{equation}\label{vu3}
\|v\|_{L^{2}(\mathbb{G})}^{2}={\rm Re}\int_{\mathbb{G}}v \overline{u} dx.
\end{equation}
Then we have
\begin{multline*}
\|u\|_{L^{2}(\mathbb{G})}^{2}-\|v\|_{L^{2}(\mathbb{G})}^{2}=
\|u\|_{L^{2}(\mathbb{G})}^{2}-\|v\|_{L^{2}(\mathbb{G})}^{2}+
2\int_{\mathbb{G}} (|v|^{2}-{\rm Re}\, v \overline{u}) dx \\
=
\int_{\mathbb{G}} (|u|^{2}+|v|^{2}-2{\rm Re}\, v \overline{u}) dx
=\int_{\mathbb{G}} |u-v|^{2} dx,
\end{multline*}
which gives \eqref{EQ:expL2}.
\end{proof}

The inequality \eqref{LpHardyinC} implies the following
uncertainly principle:

\begin{cor}[Uncertainly principle on $\mathbb{G}$]\label{Luncertainty}
Let $\mathbb{G}$ be a homogeneous group of homogeneous dimension
 $Q\geq 2$ and let $|\cdot|$ be a homogeneous quasi-norm on $\mathbb{G}$. Let $1<p<Q$ and $\frac{1}{p}+\frac{1}{q}=1$.
Then for every complex-valued function $f\in C^{\infty}_{0}(\mathbb{G}\backslash\{0\})$ we have
\begin{equation}\label{UP1}
\left(\int_{\mathbb{G}}\left|\mathcal{R}
 f\right|^{p}dx\right)^{\frac{1}{p}}
 \left(\int_{\mathbb{G}}|x|^{q}
|f|^{q}dx\right)^{\frac{1}{q}}
\geq\frac{Q-p}{p}\int_{\mathbb{G}}
|f|^{2}dx.
\end{equation}
\end{cor}
\begin{proof}
From the inequality \eqref{LpHardyeq} we get
$$
\left(\int_{\mathbb{G}}\left|\mathcal{R}
 f\right|^{p}dx\right)^{\frac{1}{p}}\left(\int_{\mathbb{G}}|x|^{q}
|f|^{q}dx\right)^{\frac{1}{q}}\geq$$
 $$\frac{Q-p}{p}\left(\int_{\mathbb{G}}
\frac{|f|^{p}}{|x|^{p}}\,dx\right)^{\frac{1}{p}}\left(\int_{\mathbb{G}}|x|^{q}
|f|^{q}dx\right)^{\frac{1}{q}}
\geq\frac{Q-p}{p}\int_{\mathbb{G}}
|f|^{2}dx,$$
where we have used the H\"older inequality in the last line.
This shows \eqref{UP1}.
\end{proof}
In the abelian case ${\mathbb G}=(\mathbb R^{n},+)$ with the standard Euclidean distance $|x|$, we have
$Q=n$, so that \eqref{UP1} with $p=q=2$ and $n\geq 3$ implies
the uncertainly principle
\begin{equation}\label{UPRn-r}
\int_{\mathbb R^{n}}\left|\frac{x}{|x|}\cdot\nabla u(x)\right|^{2}dx
\int_{\mathbb R^{n}} |x|^{2} |u(x)|^{2}dx
\geq\left(\frac{n-2}{2}\right)^{2}\left(\int_{\mathbb R^{n}}
 |u(x)|^{2} dx\right)^{2},
\end{equation}
which in turn implies
the classical
uncertainty principle for $\mathbb{G}\equiv\mathbb R^{n}$:
\begin{equation*}\label{UPRn}
\int_{\mathbb R^{n}}|\nabla u(x)|^{2}dx
\int_{\mathbb R^{n}} |x|^{2} |u(x)|^{2}dx
\geq\left(\frac{n-2}{2}\right)^{2}\left(\int_{\mathbb R^{n}}
 |u(x)|^{2} dx\right)^{2},\quad n\geq 3.
\end{equation*}

\section{Weighted Hardy inequalities}
\label{Sec4}

In this section we establish weighted Hardy inequalities on the homogeneous group
$\mathbb{G}$. This will be the consequence of the following exact formula which we believe to be new already in the setting of the Euclidean space.

\begin{thm}\label{aHardy}
Let $\mathbb{G}$ be a homogeneous group
of homogeneous dimension $Q\geq 3$ and let
$|\cdot|$ be a homogeneous quasi-norm on $\mathbb{G}$.
Then for every complex-valued function $f\in C^{\infty}_{0}(\mathbb{G}\backslash\{0\})$
we have
\begin{multline}\label{awH}
\left\|\frac{1}{|x|^{\alpha}}\mathcal{R} f\right\|^{2}_{L^{2}(\mathbb{G})}=
\left(\frac{Q-2}{2}-\alpha\right)^{2}
\left\|\frac{f}{|x|^{\alpha+1}}\right\|^{2}_{L^{2}(\mathbb{G})}
+\left\|\frac{1}{|x|^{\alpha}}\mathcal{R} f+\frac{Q-2-2\alpha}{2|x|^{\alpha+1}}f
\right\|^{2}_{L^{2}(\mathbb{G})}
\end{multline}
for all $\alpha\in\mathbb{R}.$
\end{thm}

From this we can get different inequalities. For example, if $\alpha=1$, by simplifying its coefficient we get the equality
\begin{equation}\label{47-0}
\left\|\frac{1}{|x|}\mathcal{R} f\right\|^{2}_{L^{2}(\mathbb{G})}=\left(\frac{Q-4}{2}\right)^{2}\left\|\frac{f}{|x|^{2}}
\right\|^{2}_{L^{2}(\mathbb{G})}+\left\|\frac{1}{|x|}\mathcal{R} f+\frac{Q-4}{2|x|^{2}}f
\right\|^{2}_{L^{2}(\mathbb{G})}.
\end{equation}

Now by dropping the nonnegative last term in \eqref{awH} we immediately obtain:

\begin{cor}\label{waHardy}
Let $Q\geq 3$, $\alpha\in\mathbb R$, and $Q-2\alpha-2\neq0$.
Then for all complex-valued functions $f\in C^{\infty}_{0}(\mathbb{G}\backslash\{0\})$
we have
\begin{equation}\label{awHardyeq-g}
\left\|\frac{f}{|x|^{\alpha+1}}\right\|_{L^{2}(\mathbb{G})}\leq\frac{2}{|Q-2-2\alpha|}
\left\|\frac{1}{|x|^{\alpha}}\mathcal{R} f\right\|_{L^{2}(\mathbb{G})}.
\end{equation}
The constant in \eqref{awHardyeq-g} is sharp and it is attained if and only if $f=0$.
\end{cor}

The last statement on the constant and the equality follows by the same argument as that in Remark \ref{LpHardy}. We note a special case of $\alpha=1$ again, then \eqref{47-0}
implies the estimate
\begin{equation}\label{awHardyeq}
\left\|\frac{f}{|x|^{2}}
\right\|_{L^{2}(\mathbb{G})}\leq
\frac{2}{Q-4}\left\|\frac{1}{|x|}\mathcal{R} f\right\|_{L^{2}(\mathbb{G})}, \quad Q\geq 5,
\end{equation}
again with $\frac{2}{Q-4}$ being the best constant.


If $\alpha=0$, the identity \eqref{awH} recovers Part (iii) of Theorem \ref{RemHardy}. However, we will use
Part (iii) of Theorem \ref{RemHardy} in the proof of Theorem \ref{aHardy}.

\begin{proof}[Proof of Theorem \ref{aHardy}]
First we note the equality, for $\alpha\in\mathbb{R}$,
\begin{equation}\label{EQ:eqE}
\frac{1}{|x|^{\alpha}}\mathcal{R} f=\mathcal{R} \frac{f}{|x|^{\alpha}}
+\alpha \frac{f}{|x|^{\alpha+1}}.
\end{equation}
Indeed, this follows from
$$
\mathcal{R} \frac{f}{|x|^{\alpha}}=\mathcal{R} \left(\frac{f}{|x|^{\alpha}}\right)
=\frac{1}{|x|^{\alpha}}\mathcal{R} f+f \mathcal{R}  \frac{1}{|x|^{\alpha}}
$$
and using \eqref{dfdr},
$$
\mathcal{R}  \frac{1}{|x|^{\alpha}}=\frac{d}{dr}\frac{1}{r^{\alpha}}=-\alpha\frac{1}{r^{\alpha+1}}=
-\alpha\frac{1}{|x|^{\alpha+1}},\quad r=|x|.
$$
Then we can write
\begin{multline}
\left\|\frac{1}{|x|^{\alpha}}\mathcal{R} f\right\|^{2}_{L^{2}(\mathbb{G})}=\left\|\mathcal{R} \frac{f}{|x|^{\alpha}}
+\frac{\alpha f}{|x|^{\alpha+1}}
\right\|^{2}_{L^{2}(\mathbb{G})}\\
=\left\|\mathcal{R} \frac{f}{|x|^{\alpha}}
\right\|^{2}_{L^{2}(\mathbb{G})}+2\alpha{\rm Re}\int_{\mathbb{G}}
\mathcal{R} \left(\frac{f}{|x|^{\alpha}}\right)
\frac{\overline{f}}{|x|^{\alpha+1}}dx+\left\|\frac{\alpha f}{|x|^{\alpha+1}}
\right\|^{2}_{L^{2}(\mathbb{G})}.
\end{multline}
By using \eqref{EQ:expL2} with $f$ replaced by $\frac{f}{|x|^{\alpha}}$, we have
using \eqref{EQ:eqE} that
\begin{equation}
\left\|\mathcal{R} \frac{f}{|x|^{\alpha}}
\right\|^{2}_{L^{2}(\mathbb{G})}=\left(\frac{Q-2}{2}\right)^{2}\left\|\frac{f}{|x|^{1+\alpha}}
\right\|^{2}_{L^{2}(\mathbb{G})}+\left\|\frac{1}{|x|^{\alpha}}\mathcal{R} f+\frac{Q-2-2\alpha}{2|x|^{\alpha+1}}f
\right\|^{2}_{L^{2}(\mathbb{G})}.
\end{equation}
Introducing polar coordinates $(r,y)=(|x|, \frac{x}{\mid x\mid})\in (0,\infty)\times\wp$ on $\mathbb{G}$ and using formula \eqref{EQ:polar} for polar coordinates, we have
\begin{multline}
2\alpha{\rm Re}\int_{\mathbb{G}}
\mathcal{R} \left(\frac{f}{|x|^{\alpha}}\right)
\frac{\overline{f}}{|x|^{\alpha+1}}dx=2 \alpha{\rm Re}\int_{0}^{\infty}r^{Q-2}\int_{\wp}
\frac{d}{dr}\left(\frac{f(ry)}
{r^{\alpha}}\right)\frac{\overline{f(ry)}}{r^{\alpha}}d\sigma(y)dr
\\
=\alpha\int_{0}^{\infty}r^{Q-2}\int_{\wp}
\frac{d}{dr}\left(\frac{|f(ry)|^{2}}
{r^{2\alpha}}\right)d\sigma(y)dr
=-\alpha(Q-2)\left\|\frac{f}{|x|^{\alpha+1}}
\right\|^{2}_{L^{2}(\mathbb{G})}.
\end{multline}
Summing up all above we obtain
\begin{multline*}
\left\|\frac{1}{|x|^{\alpha}}\mathcal{R} f\right\|^{2}_{L^{2}(\mathbb{G})}=
\left(\frac{Q-2}{2}-\alpha\right)^{2}\left\|\frac{f}{|x|^{\alpha+1}}
\right\|^{2}_{L^{2}(\mathbb{G})}+\left\|\frac{1}{|x|^{\alpha}}\mathcal{R} f+\frac{Q-2-2\alpha}{2|x|^{\alpha+1}}f
\right\|^{2}_{L^{2}(\mathbb{G})},
\end{multline*}
yielding \eqref{awH}.
\end{proof}

\section{Rellich inequality}
\label{Sec5}

In this section we prove the Rellich type inequality \eqref{iwHardyeq-R}. It will follow from the following identity:

\begin{thm}\label{wHardy}
Let $\mathbb{G}$ be a homogeneous group
of homogeneous dimension $Q\geq 5.$
Let $|\cdot|$ be a homogeneous quasi-norm on $\mathbb{G}$.
Then for every complex-valued function $f\in C^{\infty}_{0}(\mathbb{G}\backslash\{0\})$ we have
\begin{multline}\label{wH}
\left\|\mathcal{R} ^{2} f+\frac{Q-1}{|x|}\mathcal{R}  f
+\frac{Q(Q-4)}{4|x|^{2}}f\right\|^{2}_{L^{2}(\mathbb{G})}
+\frac{Q(Q-4)}{2}\left\|\frac{1}{|x|}\mathcal{R}  f
+\frac{Q-4}{2|x|^{2}}f\right\|^{2}_{L^{2}(\mathbb{G})}\\
=\left\|\mathcal{R} ^{2} f+\frac{Q-1}{|x|}\mathcal{R}  f
\right\|^{2}_{L^{2}(\mathbb{G})}-\left(\frac{Q(Q-4)}{4}\right)^{2}\left\|\frac{f}{|x|^{2}}
\right\|^{2}_{L^{2}(\mathbb{G})}.
\end{multline}
\end{thm}

Since the left hand side of \eqref{wH} is nonnegative we obtain the following Rellich type inequality for $Q\geq 5$:

\begin{cor}\label{weighHardy}
For all complex-valued functions $f\in C^{\infty}_{0}(\mathbb{G}\backslash\{0\})$ we have
\begin{equation}\label{wHardyeq}
\left\|\frac{f}{|x|^{2}}
\right\|_{L^{2}(\mathbb{G})}\leq\frac{4}{Q(Q-4)}\left\|\mathcal{R} ^{2} f+\frac{Q-1}{|x|}\mathcal{R}  f
\right\|_{L^{2}(\mathbb{G})},\quad Q\geq 5.
\end{equation}
The constant $\frac{4}{Q(Q-4)}$ is sharp and it is attained if and only if $f=0$.
\end{cor}
Let us argue that the constant $\frac{4}{Q(Q-4)}$ is sharp and never attained unless $f=0$.
If the equality in \eqref{wHardyeq} is attained, it follows that both terms on the left hand side of
\eqref{wH} are zero. In particular, it means that
\begin{equation}\label{EQ:aux1}
\frac{1}{|x|}\mathcal{R}  f
+\frac{Q-4}{2|x|^{2}}f=0
\end{equation}
and hence ${\tt Euler}(f)=-\frac{Q-4}{2}f$. In view of Lemma \ref{L:Euler} the function $f$ must be positively homogeneous of order $-\frac{Q-4}{2}$ which is impossible, so that the constant is not attained unless $f=0$.

Furthermore, the first term in \eqref{wH} must be also zero,
and using \eqref{EQ:aux1} this is equivalent to
$$
\mathcal{R} ^{2}f+\frac{Q-2}{2|x|} \mathcal{R} f=0
$$
which means that $\mathcal{R} f$ is positively homogeneous of order $-\frac{Q-2}{2}.$ Thus, taking an approximation of homogeneous functions $f$ of oder $-\frac{Q-4}{2}$, we have that $\mathcal{R} f$ is homogeneous of order $-\frac{Q-2}{2}$, so that the left hand side of \eqref{wH} converges to zero. Therefore, the constant $\frac{4}{Q(Q-4)}$ in \eqref{wHardyeq} is sharp.

Thus, it remains to prove \eqref{wH}.

\begin{proof}[Proof of Theorem \ref{wHardy}]
As before, introducing polar coordinates $(r,y)=(|x|, \frac{x}{\mid x\mid})\in (0,\infty)\times\wp$ on $\mathbb{G}$, using the polar decomposition formula \eqref{EQ:polar} as well as integrating by parts we obtain
\begin{multline}\label{wH4}
\int_{\mathbb{G}}
\frac{|f(x)|^{2}}
{|x|^{4}}dx
=\int_{0}^{\infty}\int_{\wp}
\frac{|f(ry)|^{2}}
{r^{4}}r^{Q-1}d\sigma(y)dr
\\
=-\frac{2}{Q-4} {\rm Re} \int_{0}^{\infty} r^{Q-4} \int_{\wp}
f(ry) \overline{\frac{df(ry)}{dr}}d\sigma(y)dr
\\
=\frac{2}{(Q-3)(Q-4)} {\rm Re} \int_{0}^{\infty} r^{Q-3} \int_{\wp}
\left( \left|\frac{df(ry)}{dr}\right|^{2}+f(ry) \overline{\frac{d^{2}f(ry)}{dr^{2}}}\right) d\sigma(y)dr
\\
=\frac{2}{(Q-3)(Q-4)}\left(\left\|\frac{1}{|x|}\mathcal{R} f\right\|^{2}_{L^{2}(\mathbb{G})}+{\rm Re}\int_{\mathbb{G}}
\frac{f(x)}{|x|^{2}}
\overline{\mathcal{R} ^{2}f(x)}dx\right).
\end{multline}
For the first term, using \eqref{awH} with $\alpha=1$, we have \eqref{47-0}, i.e.
\begin{equation}\label{47}
\left\|\frac{1}{|x|}\mathcal{R} f\right\|^{2}_{L^{2}(\mathbb{G})}=\left(\frac{Q-4}{2}\right)^{2}\left\|\frac{f}{|x|^{2}}
\right\|^{2}_{L^{2}(\mathbb{G})}+\left\|\frac{1}{|x|}\mathcal{R} f+\frac{Q-4}{2|x|^{2}}f
\right\|^{2}_{L^{2}(\mathbb{G})}.
\end{equation}
For the second term a direct calculation shows
\begin{multline}\label{48}
{\rm Re}\int_{\mathbb{G}}
\frac{f(x)}{|x|^{2}}
\overline{\mathcal{R} ^{2}f(x)}dx
\\
=
{\rm Re}\int_{\mathbb{G}}
\frac{f(x)}{|x|^{2}}
\overline{\left(\mathcal{R} ^{2}f(x)+\frac{Q-1}{|x|}\mathcal{R}  f(x)\right)}dx-
(Q-1){\rm Re}\int_{\mathbb{G}}
\frac{f(x)}{|x|^{3}}
\overline{\mathcal{R}  f(x)}dx
\\
=
{\rm Re}\int_{\mathbb{G}}
\frac{f(x)}{|x|^{2}}
\overline{\left(\mathcal{R} ^{2}f(x)+\frac{Q-1}{|x|}\mathcal{R}  f(x)\right)}dx-
\frac{Q-1}{2}\int_{0}^{\infty}r^{Q-4}\int_{\wp}
\frac{d|f(ry)|^{2}}{dr}d\sigma(y)dr
\\
={\rm Re}\int_{\mathbb{G}}
\frac{f(x)}{|x|^{2}}
\overline{\left(\mathcal{R} ^{2}f(x)+\frac{Q-1}{|x|}\mathcal{R}  f(x)\right)}dx\\
+
\frac{(Q-1)(Q-4)}{2}\int_{0}^{\infty}r^{Q-5}\int_{\wp}
|f(ry)|^{2}d\sigma(y)dr
\\
={\rm Re}\int_{\mathbb{G}}
\frac{f(x)}{|x|^{2}}
\overline{\left(\mathcal{R} ^{2}f(x)+\frac{Q-1}{|x|}\mathcal{R}  f(x)\right)}dx+
\frac{(Q-1)(Q-4)}{2}\left\|\frac{f}{|x|^{2}}
\right\|^{2}_{L^{2}(\mathbb{G})}.
\end{multline}
Combining \eqref{47} and \eqref{48} with \eqref{wH4} we arrive at
\begin{multline*}
\left\|\frac{f}{|x|^{2}}
\right\|^{2}_{L^{2}(\mathbb{G})}=\frac{2}{(Q-3)(Q-4)}
\bigg(\left(\frac{Q-4}{2}\right)^{2}\left\|\frac{f}{|x|^{2}}
\right\|^{2}_{L^{2}(\mathbb{G})}+\left\|\frac{1}{|x|}\mathcal{R} f+\frac{Q-4}{2|x|^{2}}f
\right\|^{2}_{L^{2}(\mathbb{G})}
\\
+{\rm Re}\int_{\mathbb{G}}
\frac{f(x)}{|x|^{2}}
\overline{\left(\mathcal{R} ^{2}f(x)+\frac{Q-1}{|x|}\mathcal{R}  f(x)\right)}dx+
\frac{(Q-1)(Q-4)}{2}\left\|\frac{f}{|x|^{2}}
\right\|^{2}_{L^{2}(\mathbb{G})}\bigg).
\end{multline*}
Collecting same terms, this gives
\begin{multline*}
0=\frac{2Q}{(Q-3)4}\left\|\frac{f}{|x|^{2}}
\right\|^{2}_{L^{2}(\mathbb{G})}+\frac{2}{(Q-3)(Q-4)}\bigg({\rm Re}\int_{\mathbb{G}}
\frac{f(x)}{|x|^{2}}\overline{\left(\mathcal{R} ^{2}f(x)+\frac{Q-1}{|x|}\mathcal{R}  f(x)\right)}dx
\\+\left\|\frac{1}{|x|}\mathcal{R}  f+\frac{Q-4}{2|x|^{2}}f
\right\|^{2}_{L^{2}(\mathbb{G})}\bigg),
\end{multline*}
that is,
\begin{multline*}
\frac{Q(Q-4)}{4}\left\|\frac{f}{|x|^{2}}
\right\|^{2}_{L^{2}(\mathbb{G})}= \\
-{\rm Re}\int_{\mathbb{G}}
\frac{f(x)}{|x|^{2}}\overline{\left(\mathcal{R} ^{2}f(x)+\frac{Q-1}{|x|}\mathcal{R}  f(x)\right)}dx
-\left\|\frac{1}{|x|}\mathcal{R} f+\frac{Q-4}{2|x|^{2}}f
\right\|^{2}_{L^{2}(\mathbb{G})}.
\end{multline*}
Multiplying both sides by $\frac{4}{Q(Q-4)}$ and simplifying
we obtain
\begin{multline}\label{16a}
{\rm Re}\int_{\mathbb{G}}
\frac{f(x)}{|x|^{2}}\left(\overline{\frac{f(x)}{|x|^{2}}+\frac{4}{Q(Q-4)}\left(\mathcal{R} ^{2}f(x)+\frac{Q-1}{|x|}\mathcal{R}  f(x)\right)}\right)dx \\
=
-\frac{4}{Q(Q-4)}\left\|\frac{1}{|x|}\mathcal{R} f+\frac{Q-4}{2|x|^{2}}f
\right\|^{2}_{L^{2}(\mathbb{G})}.
\end{multline}
On the other hand,
\begin{multline}\label{16b}
2{\rm Re}\int_{\mathbb{G}}
\frac{f(x)}{|x|^{2}}\left(\overline{\frac{f(x)}{|x|^{2}}+\frac{4}{Q(Q-4)}\left(\mathcal{R} ^{2}f(x)+\frac{Q-1}{|x|}\mathcal{R}  f(x)\right)}\right)dx \\
=\left\|\frac{f(x)}{|x|^{2}}+\frac{4}{Q(Q-4)}\left(\mathcal{R} ^{2}f(x)+\frac{Q-1}{|x|}\mathcal{R}  f(x)\right)
\right\|^{2}_{L^{2}(\mathbb{G})}+\left\|\frac{f(x)}{|x|^{2}}
\right\|^{2}_{L^{2}(\mathbb{G})}\\-\left\|\frac{4}{Q(Q-4)}\left(\mathcal{R} ^{2}f(x)+\frac{Q-1}{|x|}\mathcal{R}  f(x)\right)
\right\|^{2}_{L^{2}(\mathbb{G})}.
\end{multline}
From \eqref{16a} and \eqref{16b} we obtain
\begin{multline*}
-\frac{8}{Q(Q-4)}\left\|\frac{1}{|x|}\mathcal{R}  f+\frac{Q-4}{2|x|^{2}}f
\right\|^{2}_{L^{2}(\mathbb{G})}\\=\left(\frac{4}{Q(Q-4)}\right)^{2}\left\|
\frac{Q(Q-4)}{4}\frac{f(x)}{|x|^{2}}+\mathcal{R} ^{2}f(x)+\frac{Q-1}{|x|}\mathcal{R}  f(x)
\right\|^{2}_{L^{2}(\mathbb{G})}+\left\|\frac{f(x)}{|x|^{2}}
\right\|^{2}_{L^{2}(\mathbb{G})}\\
-\left(\frac{4}{Q(Q-4)}\right)^{2}\left\|\mathcal{R} ^{2}f(x)+\frac{Q-1}{|x|}\mathcal{R}  f(x)
\right\|^{2}_{L^{2}(\mathbb{G})},
\end{multline*}
thus,
\begin{multline*}
\left\|\mathcal{R} ^{2} f+\frac{Q-1}{|x|}\mathcal{R}  f
+\frac{Q(Q-4)}{4|x|^{2}}f\right\|^{2}_{L^{2}(\mathbb{G})}
+\frac{Q(Q-4)}{2}\left\|\frac{1}{|x|}\mathcal{R}  f
+\frac{Q-4}{2|x|^{2}}f\right\|^{2}_{L^{2}(\mathbb{G})}\\
=\left\|\mathcal{R} ^{2} f+\frac{Q-1}{|x|}\mathcal{R}  f
\right\|^{2}_{L^{2}(\mathbb{G})}-\left(\frac{Q(Q-4)}{4}\right)^{2}\left\|\frac{f}{|x|^{2}}
\right\|^{2}_{L^{2}(\mathbb{G})}.
\end{multline*}
The equality \eqref{wH} is proved.
\end{proof}

\section{Higher order Hardy-Rellich inequalities}
\label{SEC:ho}

In this section we show that by iterating the established weighted Hardy inequalities we get inequalities of higher order. An interesting feature is that we also obtain the exact formula for the remainder which yields the sharpness of the constants as well.

\begin{thm}\label{H-high}
Let $Q\geq 3$, $\alpha\in\mathbb{R}$ and $k\in\mathbb N$ such that $\prod_{j=0}^{k-1}
\left|\frac{Q-2}{2}-(\alpha+j)\right|\neq0$.
Then for all complex-valued functions $f\in C^{\infty}_{0}(\mathbb{G}\backslash\{0\})$
we have
\begin{equation}\label{EQ:high-order1}
\left\|\frac{f}{|x|^{k+\alpha}}
\right\|_{L^{2}(\mathbb{G})}\leq
\left[\prod_{j=0}^{k-1}
\left|\frac{Q-2}{2}-(\alpha+j)\right|\right]^{-1}\left\|\frac{1}{|x|^{\alpha}}\mathcal{R} ^{k}f\right\|_{L^{2}(\mathbb{G})},
\end{equation}
where the constant above is sharp, and is attained if and only if $f=0$.

Moreover, for all $k\in\mathbb N$ and $\alpha\in\mathbb{R}$, the following equality holds:
\begin{multline}\label{equality-high-rem}
\left\|\frac{1}{|x|^{\alpha}}\mathcal{R} ^{k}f\right\|^{2}_{L^{2}(\mathbb{G})}=
\left[\prod_{j=0}^{k-1}
\left(\frac{Q-2}{2}-(\alpha+j)\right)^{2}\right]\left\|\frac{f}{|x|^{k+\alpha}}
\right\|^{2}_{L^{2}(\mathbb{G})}
\\+\sum_{l=1}^{k-1}\left[\prod_{j=0}^{l-1}
\left(\frac{Q-2}{2}-(\alpha+j)\right)^{2}\right]\left\|\frac{1}{|x|^{l+\alpha}}\mathcal{R} ^{k-l}f+
\frac{Q-2(l+1+\alpha)}{2|x|^{l+1+\alpha}}\mathcal{R} ^{k-l-1}f\right\|^{2}_{L^{2}(\mathbb{G})}
\\
+\left\|\frac{1}{|x|^{\alpha}}\mathcal{R} ^{k}f+\frac{Q-2-2\alpha}{2|x|^{1+\alpha}}\mathcal{R} ^{k-1}f \right\|^{2}_{L^{2}(\mathbb{G})}.
\end{multline}
\end{thm}

For $k=1$ \eqref{EQ:high-order1} gives weighted Hardy's inequality and for $k=1$ and $\alpha=0$ this gives ($L^{2}$) Hardy's inequality.
For $k=2$ this can be thought of as a (weighted) Hardy-Rellich type inequality, while for larger $k$ this
corresponds to higher order (weighted) Rellich inequalities.
%
%
Although iterative methods often do not yield best constants, since we have the formula \eqref{equality-high-rem} for the remainder, we can use it to show that the iterative constant is actually sharp.
This may be a general feature of iterating Hardy-Rellich type inequalities as the same phenomena was also observed in $\Rn$ by Davies and Hinz \cite{Davies-Hinz} although they have used very different methods for their analysis.


\begin{proof}[Proof of Theorem \ref{H-high}]
We can iterate \eqref{awH}, that is, for any $\alpha\in\mathbb{R}$ and $Q\geq 3$
we have
\begin{equation}\label{awH0}
\left\|\frac{1}{|x|^{\alpha}}\mathcal{R} f\right\|^{2}_{L^{2}(\mathbb{G})}=
\left(\frac{Q-2}{2}-\alpha\right)^{2}
\left\|\frac{f}{|x|^{\alpha+1}}\right\|^{2}_{L^{2}(\mathbb{G})}
+\left\|\frac{1}{|x|^{\alpha}}\mathcal{R} f+\frac{Q-2(\alpha+1)}{2|x|^{\alpha+1}}f
\right\|^{2}_{L^{2}(\mathbb{G})}.
\end{equation}
In \eqref{awH0} replacing $f$ by $\mathcal{R} f$ we obtain
\begin{equation}\label{awH0E}
\left\|\frac{1}{|x|^{\alpha}}\mathcal{R}^{2} f\right\|^{2}_{L^{2}(\mathbb{G})}=
\left(\frac{Q-2}{2}-\alpha\right)^{2}
\left\|\frac{\mathcal{R}f}{|x|^{\alpha+1}}\right\|^{2}_{L^{2}(\mathbb{G})}
+\left\|\frac{1}{|x|^{\alpha}}\mathcal{R}^{2} f+\frac{Q-2(\alpha+1)}{2|x|^{\alpha+1}}\mathcal{R}f
\right\|^{2}_{L^{2}(\mathbb{G})}.
\end{equation}
On the other hand, replacing $\alpha$ by $\alpha+1$, \eqref{awH0} gives
\begin{multline}\label{awH1}
\left\|\frac{1}{|x|^{\alpha+1}}\mathcal{R} f\right\|^{2}_{L^{2}(\mathbb{G})}=
\left(\frac{Q-2}{2}-(\alpha+1)\right)^{2}
\left\|\frac{f}{|x|^{\alpha+2}}\right\|^{2}_{L^{2}(\mathbb{G})}
\\+\left\|\frac{1}{|x|^{\alpha+1}}\mathcal{R} f+\frac{Q-2(\alpha+2)}{2|x|^{\alpha+2}}f
\right\|^{2}_{L^{2}(\mathbb{G})}.
\end{multline}
Combining this with \eqref{awH0E} we obtain
\begin{multline}
\left\|\frac{1}{|x|^{\alpha}}\mathcal{R} ^{2}f\right\|^{2}_{L^{2}(\mathbb{G})}=
\left(\frac{Q-2}{2}-\alpha\right)^{2}
\left(\frac{Q-2}{2}-(\alpha+1)\right)^{2}
\left\|\frac{f}{|x|^{\alpha+2}}\right\|^{2}_{L^{2}(\mathbb{G})}\\
+\left(\frac{Q-2}{2}-\alpha\right)^{2}\left\|\frac{1}{|x|^{\alpha+1}}\mathcal{R} f+\frac{Q-2(\alpha+2)}{2|x|^{\alpha+2}}f
\right\|^{2}_{L^{2}(\mathbb{G})}
+\left\|\frac{1}{|x|^{\alpha}}\mathcal{R} ^{2}f+\frac{Q-2-2\alpha}{2|x|^{\alpha+1}}\mathcal{R} f
\right\|^{2}_{L^{2}(\mathbb{G})}.
\end{multline}
This iteration process gives
\begin{multline}\label{equality-high}
\left\|\frac{1}{|x|^{\alpha}}\mathcal{R} ^{k}f\right\|^{2}_{L^{2}(\mathbb{G})}=
\left[\prod_{j=0}^{k-1}
\left(\frac{Q-2}{2}-(\alpha+j)\right)^{2}\right]\left\|\frac{f}{|x|^{k+\alpha}}
\right\|^{2}_{L^{2}(\mathbb{G})}
\\+\sum_{l=1}^{k-1}\left[\prod_{j=0}^{l-1}
\left(\frac{Q-2}{2}-(\alpha+j)\right)^{2}\right]\left\|\frac{1}{|x|^{l+\alpha}}\mathcal{R} ^{k-l}f+
\frac{Q-2(l+1+\alpha)}{2|x|^{l+1+\alpha}}\mathcal{R} ^{k-l-1}f\right\|^{2}_{L^{2}(\mathbb{G})}
\\
+\left\|\frac{1}{|x|^{\alpha}}\mathcal{R} ^{k}f+\frac{Q-2-2\alpha}{2|x|^{1+\alpha}}\mathcal{R} ^{k-1}f \right\|^{2}_{L^{2}(\mathbb{G})}
, \quad k=1,2,\ldots.
\end{multline}
By dropping positive terms, it follows that
\begin{equation}\label{awHardyeq-high}
\left\|\frac{1}{|x|^{\alpha}}\mathcal{R} ^{k}f\right\|^{2}_{L^{2}(\mathbb{G})}\geq
C_{k,Q}\left\|\frac{f}{|x|^{k+\alpha}}
\right\|^{2}_{L^{2}(\mathbb{G})},
\end{equation}
where
\begin{equation}\label{constant}
C_{k,Q}=\prod_{j=0}^{k-1}
\left(\frac{Q-2}{2}-(\alpha+j)\right)^{2}.
\end{equation}
If $C_{k,Q}\neq0$,
\begin{equation}\label{awHardyeq-high}
\left\|\frac{f}{|x|^{k+\alpha}}
\right\|^{2}_{L^{2}(\mathbb{G})}\leq
\frac{1}{C_{k,Q}}\left\|\frac{1}{|x|^{\alpha}}\mathcal{R} ^{k}f\right\|^{2}_{L^{2}(\mathbb{G})}.
\end{equation}
This proves inequality \eqref{EQ:high-order1}.
For the sharpness of the constant and the equality in \eqref{EQ:high-order1}
we argue in a way similar to the sharpness argument after Corollary \ref{weighHardy}.
Namely, the equality $$\frac{\mathcal{R} ^{k-l}f}{|x|^{l+\alpha}}+
\frac{Q-2(l+1+\alpha)}{2|x|^{l+1+\alpha}}\mathcal{R} ^{k-l-1}f=0$$
can be rewritten as $$|x|\mathcal{R} (\mathcal{R} ^{k-l-1}f)+
\frac{Q-2(l+1+\alpha)}{2}(\mathcal{R} ^{k-l-1}f)=0,$$ and
by Lemma \ref{L:Euler} this means that
$\mathcal{R} ^{k-l-1}f$ is positively homogeneous of degree $-\frac{Q}{2}+l+1+\alpha.$
So, all the remainder terms vanish if $f$ is positively homogeneous of degree
$k-\frac{Q}{2}+\alpha$. As this can be approximated by functions in
$C^{\infty}_{0}(\mathbb{G}\backslash\{0\})$, the constant $C_{k,Q}$ is sharp.
Even if it were attained, it would be on functions $f$ which are positively homogeneous of
degree $k-\frac{Q}{2}+\alpha$, in which case $\frac{f}{|x|^{k+\alpha}}$ would be positively homogeneous of
degree $-\frac{Q}{2}$. These are in $L^{2}$ if and only if they are zero.
\end{proof}


\end{document}